\documentclass[twoside, reqno, 12pt]{amsart}

\usepackage[T1]{fontenc}
\usepackage[utf8]{inputenc}
\usepackage{lmodern}
\usepackage{microtype}
\usepackage{amsmath,amssymb,amsthm,mathtools}
\usepackage{geometry}
\usepackage{hyperref}
\usepackage{xcolor}

\hypersetup{colorlinks=true,linkcolor=blue!50!black,citecolor=blue!50!black,urlcolor=blue!50!black}
\allowdisplaybreaks

\makeatletter
\def\thmhead@plain#1#2#3{%
  \thmname{#1}\thmnumber{\@ifnotempty{#1}{ }\@upn{#2}}%
  \thmnote{ \the\thm@notefont(#3)}}
\let\thmhead\thmhead@plain
\makeatother

\newtheorem{theorem}{Theorem}
\newtheorem{corollary}[theorem]{Corollary}
\newtheorem{proposition}[theorem]{Proposition}
\newtheorem{lemma}[theorem]{Lemma}

\theoremstyle{definition}

\theoremstyle{remark}
\newtheorem{remark}[theorem]{Remark}

\DeclareMathOperator{\diam}{diam}

\title[A dimension-free variant of Erd\H{o}s's diameter conjecture]%
{Counterexamples to a dimension-free variant of
Erd\H{o}s's diameter conjecture for separated distances}
\author{Boon Suan Ho}
\address{Department of Mathematics, National University of Singapore, Block S17, 10 Lower Kent Ridge Road, Singapore 119076}
\email{hbs@u.nus.edu}
\date{}

\begin{document}

\begin{abstract}
Erd\H{o}s asked about the least possible diameter of an $n$-point set in
Euclidean space when any two pairwise distances differ by at least $1$.
We consider the uniform-in-dimension variant of his question, in which the
dimension~$d$ may grow with $n$ and the error term is independent of $d$, and
answer it in the negative by an explicit construction. For every prime power $q$, we construct a set
$\mathcal X_n$ of $n=q+1$ points in $\mathbb R^{q^2+q}=\mathbb R^{n^2-n}$
such that the $\binom{n}{2}$ pairwise distances are mutually at least $1$ apart, yet
\[
\diam(\mathcal X_n)\le \Bigl(1-\frac{1}{\pi^2}+o(1)\Bigr)n^2.
\]
The construction combines a Singer cyclic difference set with a weighted
product of $N=\binom n2$ regular $m$-gons, where $m=q^2+q+1$. The weights are
chosen so that the resulting distances increase while the gaps between
consecutive distances decrease. After normalizing the smallest gap to $1$, all
distinct pairwise distances are separated by at least $1$, while the diameter
has asymptotic constant $1-1/\pi^2<1$.
\end{abstract}
\maketitle
\section{Introduction}
For integers $d\ge1$ and $n\ge3$, let $\Delta_d(n)$ denote the infimum of
$\diam(\mathcal A)$ over all $n$-point sets $\mathcal A\subset\mathbb R^d$
such that
\[
\bigl|\,\|x-y\|-\|x'-y'\|\,\bigr|\ge1
\]
whenever $\{x,y\}\ne\{x',y'\}$ are unordered pairs from $\mathcal A$.
Evidently, $\diam(\mathcal A)\ge\binom n2$ for every such $\mathcal A$: for
any pair $\{x,y\}$ and any $z\in\mathcal A\setminus\{x,y\}$, the triangle
inequality gives $\|x-y\|\ge\bigl|\,\|x-z\|-\|y-z\|\,\bigr|\ge1$, so the
$\binom n2$ pairwise distances are all at least $1$ and mutually at least $1$
apart. Hence $\Delta_d(n)\ge\binom n2$.
In a chapter of unsolved problems, Erd\H{o}s wrote that a conjecture
``independent of dimension'' is that the diameter is at least
$(1+o(1))n^2$ \cite[Problem 20]{Erdos1997}. His formulation does not specify
the order of the quantifiers, and there are two plausible readings:
\begin{align}
\Delta_d(n)&\ge (1+o_d(1))n^2
&&\text{as $n\to\infty$, for each fixed $d$,} \tag{A}\label{eq:fixed-d}\\
\inf_{d\ge1}\Delta_d(n)&\ge (1+o(1))n^2
&&\text{as $n\to\infty$.} \tag{B}\label{eq:uniform-d}
\end{align}
In \eqref{eq:fixed-d}, the error term may depend on $d$; in
\eqref{eq:uniform-d}, it is uniform in the dimension, so $d$ may grow with
$n$.
Although the phrase ``independent of dimension'' may suggest
\eqref{eq:uniform-d}, the surrounding discussion---in particular, Erd\H{o}s's
statement that the conjecture was known only in $\mathbb R$ and not even in
$\mathbb R^2$---suggests that \eqref{eq:fixed-d} was the intended question.
The purpose of this note is to disprove only the uniform-in-dimension
statement \eqref{eq:uniform-d}.

To the best of our knowledge, the exact higher-dimensional form of
Erd\H{o}s's question has received relatively little direct attention. The
closest earlier work seems to be Brass's study of the planar
``Erd\H{o}s-diameter'' \cite{Brass1996}, which assumes that distinct positive
distances are separated by at least $1$ and asks for the minimum possible
diameter. There is also a substantial literature on nearly equal distances,
beginning with Erd\H{o}s--Makai--Pach--Spencer \cite{EMPS1991} and
Erd\H{o}s--Makai--Pach \cite{EMP1993}, in which many pairwise distances are
allowed to lie in one or several short intervals. See also
\cite{FranklKupavskii2023, PRV2006}. These problems are different
from the present one, but they are motivated by the same general question:
how strongly does the spacing of the distance set constrain the geometry of the
underlying point configuration?

Our main result is the following counterexample to
\eqref{eq:uniform-d}.

\begin{theorem}\label{thm:main}
Let
\[
 c_*\coloneqq1-\frac{1}{\pi^2}=0.89867\dots.
\]
For infinitely many $n$, there exists an $n$-point set
$\mathcal X_n\subset \mathbb R^{n^2-n}$ such that
\begin{equation}\label{eq:sep-condition}
\left|\,\|x-y\|-\|x'-y'\|\,\right|\ge 1
\end{equation}
for every two distinct pairs
$\{x,y\}\neq \{x',y'\}$ from $\mathcal X_n$, while
\[
\diam(\mathcal X_n)\le (c_*+o(1))n^2.
\]
\end{theorem}

The dimension in Theorem \ref{thm:main} grows with $n$. Consequently, the
theorem disproves the uniform statement \eqref{eq:uniform-d}, but it does not
address the fixed-dimensional conjecture \eqref{eq:fixed-d}. Together with
the bound $\Delta_d(n)\ge\binom n2$ noted above, it shows that
\[
\frac12\le\liminf_{n\to\infty}\frac{\inf_{d\ge1}\Delta_d(n)}{n^2}\le c_*.
\]
The explicit
coordinates below lie in $\mathbb R^{n^2-n}$; as with every $n$-point
Euclidean configuration, the same distance data can also be realized in an
affine subspace of dimension at most $n-1$.

The construction is explicit. Let $q$ be a prime power and put
\[
 m=q^2+q+1,
 \qquad n=q+1,
 \qquad N=\frac{m-1}{2}=\binom{n}{2}.
\]
Singer's theorem gives a cyclic $(m,n,1)$-difference set
$D\subset \mathbb Z/m\mathbb Z$ \cite{Singer1938}: this is a set of $n$
residues for which every nonzero residue modulo $m$ occurs exactly once as an
ordered difference of two elements of $D$. Thus the $N$ unordered pairs from
$D$ are indexed by the cyclic separations $1,\dots,N$. We place the points in
a weighted product of $N$ regular $m$-gons and choose the weights so that the
distance $d_s$ associated with separation $s$ satisfies
$d_1<\dots<d_N$. We also arrange that the gaps $d_{s+1}-d_s$ decrease with
$s$; then the smallest gap is the last one, and scaling by its reciprocal
makes all pairwise distances at least $1$ apart.

\section{Singer difference sets and the Euclidean construction}

We begin with the classical cyclic difference set construction.

\begin{proposition}[Singer \cite{Singer1938}]\label{prop:singer}
Let $q$ be a prime power and let $m=q^2+q+1$. There exists a set
$D\subset \mathbb Z/m\mathbb Z$ with $|D|=q+1$ such that every nonzero residue class in
$\mathbb Z/m\mathbb Z$ has a unique representation in the form $a-b$ with $a,b\in D$.
\end{proposition}

Fix such a set $D$, and write $n=q+1$ and $N=(m-1)/2$. For an unordered pair $\{t,u\}\subset \mathbb Z/m\mathbb Z$ with $t\neq u$, its
\emph{cyclic separation} is the unique integer $s\in[N]\coloneqq\{1,\dots,N\}$ such that
$t-u\equiv \pm s \pmod m$.

\begin{lemma}\label{lem:separations}
For each $s\in[N]$, there is exactly one unordered pair
$\{a,b\}\subset D$ such that
\[
a-b\equiv \pm s \pmod m.
\]
Consequently, the $\binom n2=N$ unordered pairs from $D$ are indexed by
$s=1,\dots,N$.
\end{lemma}

\begin{proof}
By Proposition \ref{prop:singer}, for each nonzero residue $r\in\mathbb Z/m\mathbb Z$,
there is a unique ordered pair $(a,b)\in D\times D$ with $a-b\equiv r\pmod m$.
The residues $s$ and $-s$ correspond to the same unordered pair, with the order
reversed. Thus each $s\in[N]$ determines at most one unordered pair.
Since there are $N$ values of $s$ and exactly $\binom n2=N$ unordered pairs,
each value occurs exactly once.
\end{proof}

\begin{proposition}\label{prop:distance-realization}
Let $W_1,\dots,W_N\ge 0$. For $t\in\mathbb Z/m\mathbb Z$ define
\[
 \gamma_r(t)\coloneqq\sqrt{\frac{W_r}{2}}
 \left(\cos\frac{2\pi rt}{m},\sin\frac{2\pi rt}{m}\right)
 \in\mathbb R^2
 \qquad (r\in[N])
 \]
and
\[
 v_t\coloneqq\bigl(\gamma_1(t),\dots,\gamma_N(t)\bigr)
 \in(\mathbb R^2)^N\cong\mathbb R^{2N}.
\]
If $\{t,u\}$ has cyclic separation $s\in[N]$, then
\begin{equation}\label{eq:distance-formula-general}
\|v_t-v_u\|^2
=\sum_{r=1}^N W_r\Bigl(1-\cos\frac{2\pi rs}{m}\Bigr).
\end{equation}
Hence, the multiset of pairwise distances among $\{v_t:t\in D\}$ is exactly the
collection of values
\[
 d_s\coloneqq\left(\sum_{r=1}^N W_r\Bigl(1-\cos\frac{2\pi rs}{m}\Bigr)\right)^{1/2},
 \qquad s=1,\dots,N,
\]
with each value (not necessarily distinct) occurring once.
\end{proposition}

\begin{proof}
In the $r$th two-dimensional factor, the squared distance between the $t$th and
$u$th points equals
\begin{align*}
&\frac{W_r}{2}\left(\bigl(\cos(2\pi rt/m)-\cos(2\pi ru/m)\bigr)^2
+\bigl(\sin(2\pi rt/m)-\sin(2\pi ru/m)\bigr)^2\right) \\
&\qquad= W_r\Bigl(1-\cos\frac{2\pi r(t-u)}{m}\Bigr)
\end{align*}
by the identity $\cos(a-b)=\cos(a)\cos(b)+\sin(a)\sin(b)$. Summing over $r$
yields \eqref{eq:distance-formula-general}. The final statement
follows from Lemma \ref{lem:separations}.
\end{proof}

\section{Choosing the weights}

We now choose the weights of the $N$ regular $m$-gons in Proposition
\ref{prop:distance-realization}. For a pair with cyclic separation $s$, write
$\theta=2\pi s/m$. Suppose first that the coefficient attached to each
positive odd integer $j$ is $1/j^2$. Then the square of the corresponding
distance is
\[
\sum_{\substack{j\ge1\\j\ \mathrm{odd}}}
\frac{1-\cos(j\theta)}{j^2}=\frac{\pi\theta}{4}
\qquad (0\le\theta\le\pi),
\]
by the Fourier-series identity verified below. We reduce only the coefficient
for $j=1$, from $1$ to $7/8$. We will show that the resulting distances
increase with $s$, while the gaps between consecutive distances decrease.

Fix $\varepsilon\coloneqq1/8$ and define
\[
 c_1\coloneqq1-\varepsilon=\frac78
 \quad\text{and}\quad
 c_j\coloneqq\frac1{j^2}
\]
for odd $j\ge3$. For $r\in[N]$, set
\begin{equation}\label{eq:Wr}
W_r\coloneqq\sum_{\substack{j\ge1,\ j\ \mathrm{odd}\\
j\equiv \pm r\!\!\!\pmod m}} c_j.
\end{equation}
Then $W_r>0$ for every $r$ (indeed, because $m$ is odd, exactly one of $r$ and
$m-r$ is an odd positive integer congruent to $\pm r \pmod m$, so the sum in
\eqref{eq:Wr} contains at least one positive term). Applying Proposition
\ref{prop:distance-realization} with these weights gives a point set
\[
\mathcal Y_n\coloneqq\{v_t:t\in D\}\subset \mathbb R^{2N}=\mathbb R^{m-1}.
\]

\begin{proposition}\label{prop:explicit-distances}
Let
\[
H\colon[0,\pi]\longrightarrow\mathbb R,\qquad
H(\theta)\coloneqq
\sum_{\substack{j\ge1\\j\ \mathrm{odd}}}c_j(1-\cos(j\theta))
\]
and
\[
G\colon(0,\pi]\longrightarrow(0,\infty),\qquad
G(\theta)\coloneqq\sqrt{H(\theta)}.
\]
Then the pairwise distances in $\mathcal Y_n$ are exactly
\begin{equation}\label{eq:ds-grid}
 d_s=G(2\pi s/m), \qquad s=1,\dots,N.
\end{equation}
Moreover,
\begin{equation}\label{eq:H-explicit}
H(\theta)=\frac{\pi\theta}{4}-\varepsilon(1-\cos\theta)
\qquad (0\le \theta\le \pi).
\end{equation}
\end{proposition}

\begin{proof}
By Proposition \ref{prop:distance-realization}, the square of the distance
associated with cyclic separation $s$ is
\[
\sum_{r=1}^N W_r\left(1-\cos\frac{2\pi rs}{m}\right).
\]
If $j\equiv \pm r\pmod m$, then for every integer $s$ we have
\[
\cos\frac{2\pi js}{m}=\cos\frac{2\pi rs}{m}.
\]
Therefore, for each $s\in[N]$,
\[
\sum_{r=1}^N W_r\Bigl(1-\cos\frac{2\pi rs}{m}\Bigr)
=\sum_{\substack{j\ge1,\ j\ \mathrm{odd}\\
j\not\equiv0\pmod m}}c_j\Bigl(1-\cos\frac{2\pi js}{m}\Bigr).
\]
The terms with $j\equiv 0\pmod m$ contribute $0$, so by absolute convergence
this equals
\[
\sum_{\substack{j\ge1\\j\ \mathrm{odd}}}
c_j\Bigl(1-\cos\frac{2\pi js}{m}\Bigr)
=H(2\pi s/m),
\]
which proves \eqref{eq:ds-grid}.
To obtain \eqref{eq:H-explicit}, write
\[
F(\theta)\coloneqq\sum_{k\ge 1}\frac{1-\cos(k\theta)}{k^2}
=\frac{\pi\theta}{2}-\frac{\theta^2}{4}
\qquad (0\le \theta\le 2\pi)
\]
using a standard Fourier series evaluation. The odd part is
\[
\sum_{\substack{j\ge1\\j\ \mathrm{odd}}}
\frac{1-\cos(j\theta)}{j^2}
=F(\theta)-\frac14F(2\theta)
=\frac{\pi\theta}{4}
\qquad (0\le \theta\le \pi).
\]
Since the coefficients $c_j$ differ from $1/j^2$ only at $j=1$, equation
\eqref{eq:H-explicit} follows.
\end{proof}

\begin{lemma}\label{lem:monotone-concave}
The function $G$ is strictly increasing and strictly concave on $(0,\pi)$.
\end{lemma}

\begin{proof}
By \eqref{eq:H-explicit},
\[
H'(\theta)=\frac{\pi}{4}-\varepsilon\sin\theta
\ge\frac{\pi}{4}-\frac18>0
\qquad (0<\theta<\pi).
\]
Thus $H$, and hence $G$, is strictly increasing on $(0,\pi)$.

For concavity, set
\[
Q(\theta)\coloneqq2H(\theta)H''(\theta)-H'(\theta)^2
\qquad (0<\theta\le\pi).
\]
Then
\begin{equation}\label{eq:G-second}
G''(\theta)=\frac{Q(\theta)}{4H(\theta)^{3/2}}.
\end{equation}
Since $H'''(\theta)=\varepsilon\sin\theta$, differentiation gives
\[
Q'(\theta)=2H(\theta)H'''(\theta)
=2\varepsilon H(\theta)\sin\theta>0
\qquad (0<\theta<\pi).
\]
Therefore $Q(\theta)<Q(\pi)$. Using $\varepsilon=1/8$, we obtain
\[
Q(\pi)
=2\left(\frac{\pi^2-1}{4}\right)\frac18
-\left(\frac{\pi}{4}\right)^2
=-\frac1{16}<0.
\]
It follows from \eqref{eq:G-second} that $G''(\theta)<0$ for
$0<\theta<\pi$.
\end{proof}

\begin{corollary}\label{cor:smallest-gap}
The sequence
\[
 d_s=G(2\pi s/m),\qquad s=1,\dots,N,
\]
is strictly increasing and satisfies
\[
 d_{s+1}-d_s > d_{s+2}-d_{s+1}
 \qquad (1\le s\le N-2).
\]
In particular, the smallest gap is $d_N-d_{N-1}$.
\end{corollary}

\begin{proof}
By Lemma \ref{lem:monotone-concave}, $d_s$ is strictly increasing.
Put $\theta_s=2\pi s/m$. Since
$\theta_{s+1}=(\theta_s+\theta_{s+2})/2$, strict concavity gives
\[
d_{s+1}=G(\theta_{s+1})
>\frac{G(\theta_s)+G(\theta_{s+2})}{2}
=\frac{d_s+d_{s+2}}{2}.
\]
Rearranging yields
$d_{s+1}-d_s>d_{s+2}-d_{s+1}$.
\end{proof}

Let
\[
\lambda_n\coloneqq\frac{1}{d_N-d_{N-1}},
\qquad
\mathcal X_n\coloneqq\lambda_n\mathcal Y_n.
\]
By Corollary \ref{cor:smallest-gap}, every two distinct distances in
$\mathcal X_n$ differ by at least $1$. Thus $\mathcal X_n$ satisfies
\eqref{eq:sep-condition}, and
\begin{equation}\label{eq:diam-lambda}
\diam(\mathcal X_n)=\lambda_n d_N=\frac{d_N}{d_N-d_{N-1}}.
\end{equation}

\section{Asymptotic diameter}

We now evaluate \eqref{eq:diam-lambda}.

\begin{proposition}\label{prop:asymptotic-diameter}
Let $m=q^2+q+1$ and $n=q+1$, with $q$ a prime power. Then
\[
\diam(\mathcal X_n)=\left(1-\frac{1}{\pi^2}+o(1)\right)m
=\left(1-\frac{1}{\pi^2}+o(1)\right)n^2.
\]
\end{proposition}

\begin{proof}
Set
\[
\theta_N\coloneqq\frac{2\pi N}{m}=\pi-\frac{\pi}{m},
\qquad
\theta_{N-1}\coloneqq\frac{2\pi(N-1)}{m}=\pi-\frac{3\pi}{m}.
\]
By the mean value theorem,
\[
 d_N-d_{N-1}=G(\theta_N)-G(\theta_{N-1})
 =\frac{2\pi}{m}G'(\xi_m)
\]
for some $\xi_m\in(\theta_{N-1},\theta_N)$. Since $\xi_m\to\pi$ and
$\theta_N\to\pi$, equation \eqref{eq:diam-lambda} yields
\[
\frac{\diam(\mathcal X_n)}{m}
=\frac{G(\theta_N)}{2\pi G'(\xi_m)}
\longrightarrow \frac{G(\pi)}{2\pi G'(\pi^-)},
\]
where $G'(\pi^-)\coloneqq\lim_{x\to\pi^-}G'(x)$. Now
\[
H(\pi)=\frac{\pi^2}{4}-2\varepsilon=\frac{\pi^2-1}{4},
\qquad
H'(\pi)=\frac{\pi}{4},
\]
so
\[
G'(\pi^-)=\frac{H'(\pi)}{2G(\pi)}=\frac{\pi}{8G(\pi)}.
\]
Therefore
\[
\frac{G(\pi)}{2\pi G'(\pi^-)}
=\frac{4H(\pi)}{\pi^2}
=1-\frac{1}{\pi^2}.
\]
Since $m=n^2-n+1$, the same asymptotic constant holds with $m$ replaced by
$n^2$.
\end{proof}

\begin{proof}[Proof of Theorem \ref{thm:main}]
Take $n=q+1$ with $q$ a prime power and construct $\mathcal X_n$ as above. Then
$\mathcal X_n\subset \mathbb R^{m-1}=\mathbb R^{n^2-n}$ has $n$ points,
satisfies \eqref{eq:sep-condition}, and has diameter bounded as in Proposition%
~\ref{prop:asymptotic-diameter}. Since there are infinitely many prime powers,
this provides infinitely many such values of $n$. The strict inequality
$c_*<1$ violates the uniform-in-dimension statement \eqref{eq:uniform-d}.
\end{proof}

\section{Remarks}

\begin{remark}
The fixed-dimensional conjecture \eqref{eq:fixed-d} remains open. Explicitly,
for each fixed $d\ge2$, one asks whether every $n$-point set in $\mathbb R^d$
whose $\binom n2$ pairwise distances are mutually at least $1$ apart has
diameter at least $(1+o_d(1))n^2$ as $n\to\infty$. The notation $o_d(1)$
means that the error may depend on $d$, but must tend to $0$ after $d$ is
fixed. Theorem \ref{thm:main} gives no information about this question because
its dimension tends to infinity with $n$.
\end{remark}

\begin{remark}
More generally, replacing $c_1=7/8$ by $1-\varepsilon$ and keeping
$c_j=1/j^2$ for odd $j\ge3$, the same construction yields
\[
H_\varepsilon(\theta)=\frac{\pi\theta}{4}
-\varepsilon(1-\cos\theta)
\]
and asymptotic constant $1-8\varepsilon/\pi^2$.
For $0<\varepsilon<\pi/4$, the function $H_\varepsilon$ is strictly
increasing. If
\[
Q_\varepsilon(\theta)
\coloneqq2H_\varepsilon(\theta)H_\varepsilon''(\theta)
-H_\varepsilon'(\theta)^2,
\]
then
\[
Q_\varepsilon'(\theta)
=2\varepsilon H_\varepsilon(\theta)\sin\theta>0
\qquad (0<\theta<\pi).
\]
Thus $\sqrt{H_\varepsilon}$ is concave on $(0,\pi)$ precisely when
$Q_\varepsilon(\pi)\le0$. In the relevant range this is equivalent to
\[
0<\varepsilon\le
\varepsilon_*\coloneqq
\frac{\pi}{16}\bigl(\pi-\sqrt{\pi^2-4}\bigr),
\]
which gives the better constant
\[
1-\frac{8\varepsilon_*}{\pi^2}
=\frac{\pi+\sqrt{\pi^2-4}}{2\pi}
=0.885589\ldots.
\]
We do not pursue further optimization of the coefficients $c_j$ here.
\end{remark}

\section*{Statements and Declarations}

\textbf{Funding.} No funding was received for conducting this study or for the preparation of this manuscript.

\medskip
\noindent\textbf{Competing interests.} The author has no competing interests to declare that are relevant to the content of this article.

\medskip
\noindent\textbf{Data availability.} No datasets were generated or analysed in the current study. A Lean~4 formalization of the proofs and supporting source files are available at \url{https://github.com/boonsuan/erdos670}.

\medskip
\noindent\textbf{Code availability.} The Lean~4 formalization accompanying this paper is available at \url{https://github.com/boonsuan/erdos670}.

\section*{Use of AI tools}

GPT-5.4 Pro was used during exploratory work that led to the construction
presented here and to assist with parts of the Lean~4 formalization; Harmonic
Aristotle was also used in the formalization. All mathematical arguments and
claims in the final manuscript were independently verified by the author, who
takes full responsibility for the paper.

\section*{Acknowledgements}
The author thanks Thomas Bloom for a helpful discussion about the quantifiers
in Erd\H{o}s's formulation, and Win Way Yan and Alyxia Seah for helpful
comments on a draft of this paper.

\end{document}